\documentclass[times]{oupau-ppn}

\usepackage[T1]{fontenc}
\usepackage[utf8]{inputenc}
\usepackage{amsmath,amsfonts,amsthm,amssymb,amsxtra,bbm}
\usepackage{fourier,setspace,graphicx,color,pdflscape}
\usepackage[colorlinks=true]{hyperref}
\hypersetup{urlcolor=blue, linkcolor=blue, citecolor=red, anchorcolor=blue]}
\usepackage[numbers]{natbib}

\newtheorem{remark}[theorem]{Remark}
\newcommand{\R}{{\mathbb R}}

\newcommand{\be}[1]{\begin{equation}\label{#1}}
\newcommand{\ee}{\end{equation}}
\renewcommand{\(}{\left(}
\renewcommand{\)}{\right)}

\newcommand{\nrm}[2]{\left\|{#1}\right\|_{#2}}
\newcommand{\ix}[1]{\int_{\R^2}{#1}\,dx}

\begin{document}
\volumeyear{}\paperID{}

\title{\Large Generalized logarithmic Hardy-Littlewood-Sobolev inequality}
\shorttitle{Generalized logarithmic HLS inequality}

\author{Jean Dolbeault\affil{1} and Xingyu Li\affil{1}}
\abbrevauthor{J.~Dolbeault and X.~Li}
\headabbrevauthor{Dolbeault, J., and Li, X.}

\address{\affilnum{1} CEREMADE (CNRS UMR n$^\circ$ 7534), PSL university, Universit\'e Paris-Dauphine,\\ Place de Lattre de Tassigny, 75775 Paris 16, France}
\correspdetails{\href{mailto:dolbeaul@ceremade.dauphine.fr}{dolbeaul@ceremade.dauphine.fr}}
\begin{abstract} 
This paper is devoted to logarithmic Hardy-Littlewood-Sobolev inequalities in the two-dimensional Euclidean space, in presence of an external potential with logarithmic growth. The coupling with the potential introduces a new parameter, with two regimes. The attractive regime reflects the standard logarithmic Hardy-Littlewood-Sobolev inequality. The second regime corresponds to a reverse inequality, with the opposite sign in the convolution term, which allows us to bound the free energy of a drift-diffusion-Poisson system from below. Our method is based on an extension of an entropy method proposed by E. Carlen, J. Carrillo and M. Loss, and on a nonlinear diffusion equation.\end{abstract}
\maketitle
\thispagestyle{empty}
\footnotetext{\emph{Keywords:} 
logarithmic Hardy-Littlewood-Sobolev inequality; drift-diffusion-Poisson equation; entropy methods; nonlinear parabolic equations}
\footnotetext{\emph{MSC 2010:} 26D10; 46E35; 35K55}
\section{Main result and motivation}\label{Sec:Intro}

On $\R^2$, let us define the \emph{density of probability} $\mu=e^{-V}$ and the \emph{external potential} $V$ by
\[
\mu(x):=\frac 1{\pi\,\big(1+|x|^2\big)^2}\quad\mbox{and}\quad V(x):=-\,\log\mu(x)=2\,\log\(1+|x|^2\)+\log\pi\quad\forall\,x\in\R^2\,.
\]
We shall denote by $\mathrm L^1_+(\R^2)$ the set of a.e. nonnegative functions in $\mathrm L^1(\R^2)$. Our main result is the following \emph{generalized logarithmic Hardy-Littlewood-Sobolev inequality}.
\begin{theorem}\label{Thm:Main} For any $\alpha\ge0$, we have that
\be{Ineq:LogHLS}
\ix{f\,\log\(\frac fM\)}+\alpha\ix{V\,f}+M\,(1-\alpha)\,\(1+\log\pi\)\ge\frac 2M\,(\alpha-1)\iint_{\R^2\times\R^2}f(x)\,f(y)\,\log|x-y|\,dx\,dy
\ee
for any function $f\in\mathrm L^1_+(\R^2)$ with $M=\ix f>0$. Moreover, the equality case is achieved by $f_\star=M\,\mu$ and~$f_\star$ is the unique optimal function for any $\alpha>0$.\end{theorem}
With $\alpha=0$, the inequality is the classical \emph{logarithmic Hardy-Littlewood-Sobolev inequality}
\be{logHLS0}
\ix{f\,\log\(\frac fM\)}+\frac 2M\iint_{\R^2\times\R^2}f(x)\,f(y)\,\log|x-y|\,dx\,dy+M\,\(1+\log\pi\)\ge0\,.
\ee
In that case $f_\star$ is an optimal function as well as all functions generated by a translation and a scaling of $f_\star$. As long as the parameter $\alpha$ is in the range $0\le\alpha<1$, the coefficient of the right-hand side of~\eqref{Ineq:LogHLS} is negative and the inequality is essentially of the same nature as the one with $\alpha=0$. It can indeed be written as
\[
\ix{f\,\log\(\frac fM\)}+\alpha\ix{V\,f}+M\,(1-\alpha)\,\(1+\log\pi\)+\frac 2M\,(1-\alpha)\iint_{\R^2\times\R^2}f(x)\,f(y)\,\log|x-y|\,dx\,dy\ge0\,.
\]
For reasons that will be made clear below, we shall call this range the \emph{attractive range}.

If $\alpha=1$, the inequality is almost trivial since
\be{Jensen1}
\ix{f\,\log\(\frac fM\)}+\ix{V\,f}=\ix{f\,\log\(\frac f{f_\star}\)}\ge0
\ee
is a straightforward consequence of Jensen's inequality. Now it is clear that by adding~\eqref{logHLS0} multiplied by $(1-\alpha)$ and~\eqref{Jensen1} multiplied by $\alpha$, we recover~\eqref{Ineq:LogHLS} for any $\alpha\in[0,1]$. As a consequence~\eqref{Ineq:LogHLS} is a straightforward interpolation between~\eqref{logHLS0} and~\eqref{Jensen1} in the \emph{attractive range}.

Now, let us consider the \emph{repulsive range} $\alpha>1$. It is clear that the inequality is no more the consequence of a simple interpolation. We can also observe that the coefficient $(\alpha-1)$ in the right-hand side of~\eqref{Ineq:LogHLS} is now positive. Since
\[
G(x)=-\,\frac1{2\,\pi}\,\log|x|
\]
is the Green function associated with $-\,\Delta$ on $\R^2$, so that we can define
\[
(-\Delta)^{-1}f(x)=(G*f)(x)=-\,\frac1{2\,\pi}\int_{\R^2}\log|x-y|\,f(y)\,dy\,,
\]
it is interesting to write~\eqref{Ineq:LogHLS} as
\be{logHLSrep}
\ix{f\,\log\(\frac fM\)}+\alpha\ix{V\,f}+\frac{4\,\pi}M\,(\alpha-1)\ix{f\,(-\Delta)^{-1}f}\ge M\,(\alpha-1)\,\(1+\log\pi\)\,.
\ee
If $f$ has a sufficient decay as $|x|\to+\infty$, for instance if $f$ is compactly supported, we know that $(-\Delta)^{-1}f(x)\sim-\,\frac M{2\,\pi}\,\log|x|$ for large values of $|x|$ and as a consequence, 
\[
\alpha\,V+\frac{4\,\pi}M\,(\alpha-1)\,(-\Delta)^{-1}f\sim2\,(\alpha+1)\,\log|x|\to+\infty\quad\mbox{as}\quad|x|\to+\infty\,.
\]
In a minimization scheme, this prevents the runaway of the left-hand side in~\eqref{logHLSrep}. On the other hand, $\ix{f\,\log f}$ prevents any concentration, and this is why it can be heuristically expected that the left-hand side of~\eqref{logHLSrep} indeed admits a minimizer.

Inequality~\eqref{logHLS0} was proved in~\cite{MR1143664} by E.~Carlen and M.~Loss (also see~\cite{MR1230930}). An alternative method based on nonlinear flows was given by E.~Carlen, J.~Carrillo and M.~Loss in~\cite{MR2745814}: see Section~\ref{Sec:Proof} for a sketch of their proof. Our proof of Theorem~\ref{Thm:Main} relies on an extension of this approach which takes into account the presence of the external potential $V$. A remarkable feature of this approach is that it is insensitive to the sign of $\alpha-1$.

One of the key motivations for studying~\eqref{logHLSrep} arises from entropy methods applied to \emph{drift-diffusion-Poisson} models which, after scaling out all physical parameters, are given by
\be{DD}
\frac{\partial f}{\partial t}=\Delta f+\beta\,\nabla\cdot(f\,\nabla V)+\nabla\cdot(f\,\nabla\phi)
\ee
with a nonlinear coupling given by the \emph{Poisson} equation
\be{P}
-\,\varepsilon\,\Delta\phi=f\,.
\ee
Here $V=-\,\log\mu$ is the external \emph{confining potential} and we choose it as in the statement of Theorem~\ref{Thm:Main}, while $\beta\ge0$ is a coupling parameter with $V$, which measures the strength of the external potential. We shall consider more general potentials at the end of this paper. The coefficient $\varepsilon$ in~\eqref{P} is either $\varepsilon=-1$, which corresponds to the \emph{attractive} case, or $\varepsilon=+1$, which corresponds to the \emph{repulsive} case. In terms of applications, when $\varepsilon=-1$,~\eqref{P} is the equation for the mean field potential obtained from Newton's law of attraction in gravitation, for applications in astrophysics, or for the Keller-Segel concentration of chemo-attractant in chemotaxis. The case $\varepsilon=+1$ is used for repulsive electrostatic forces in semi-conductor physics, electrolytes, plasmas and charged particle models.

In view of \emph{entropy methods} applied to PDEs (see for instance~\cite{MR3497125}), it is natural to consider the \emph{free energy functional}
\be{FreeEnergy}
\mathcal F_\beta[f]:=\ix{f\,\log f}+\beta\ix{V\,f}+\frac12\ix{\phi\,f}
\ee
because, if $f>0$ solves~\eqref{DD}-\eqref{P} and is smooth enough, with sufficient decay properties at infinity, then
\be{EntropyProduction}
\frac d{dt}\mathcal F_\beta[f(t,\cdot)]=-\ix{f\,\left|\nabla\log f+\beta\,\nabla V+\nabla\phi\right|^2}
\ee
so that $\mathcal F_\beta$ is a Lyapunov functional. Of course, a preliminary question is to establish under which conditions $\mathcal F_\beta$ is bounded from below. The answer is given by the following result.
\begin{corollary}\label{Cor:Main} Let $M>0$. If $\varepsilon=+1$, the functional $\mathcal F_\beta$ is bounded from below and admits a minimizer on the set of the functions $f\in\mathrm L^1_+(\R^2)$ such that $\ix f=M$ if $\beta\ge1+\frac M{8\,\pi}$. It is bounded from below if $\varepsilon=-1$, $\beta\ge1-\frac M{8\,\pi}$ and $M\le8\,\pi$. If $\varepsilon=+1$, the minimizer is unique.\end{corollary}
As we shall see in Section~\ref{Sec:ProofCorollary}, Corollary~\ref{Cor:Main} is a simple consequence of Theorem~\ref{Thm:Main}. In the case of the parabolic-elliptic Keller-Segel model, that is, with $\varepsilon=-1$ and $\beta=0$, this has been used in~\cite{MR2103197,MR2226917} to provide a sharp range of existence of the solutions to the evolution problem. In~\cite{MR3196188}, the case $\varepsilon=-1$ with a potential $V$ with quadratic growth at infinity was also considered, in the study of intermediate asymptotics of the parabolic-elliptic Keller-Segel model.

\medskip Concerning the \emph{drift-diffusion-Poisson} model~\eqref{DD}-\eqref{P} and considerations on the \emph{free energy}, in the electrostatic case, we can quote, among many others,~\cite{MR1015923,MR1058151} and subsequent papers. In the Euclidean space with confinig potentials, we shall refer to \cite{MR1115290,MR1677677,MR1777308,MR1842428}. However, as far as we know, these papers are primarily devoted to dimensions $d\ge3$ and the sharp growth condition on $V$ when $d=2$ has not been studied so far. The goal of this paper is to fill this gap. The specific choice of $V$ has been made to obtain explicit constants and optimal inequalities, but the confining potential plays a role only at infinity if we are interested in the boundedness from below of the free energy. In Section~\ref{Sec:confinement}, we shall give a result for general potentials on $\R^2$: see Theorem~\ref{Thm:confinement} for a statement.

\section{Proof of the main result}\label{Sec:Proof}

As an introduction to the key method, we briefly sketch the proof of~\eqref{logHLS0} given by E.~Carlen, J.~Carrillo and M.~Loss in~\cite{MR2745814}. The main idea is to use the nonlinear diffusion equation
\[
\frac{\partial f}{\partial t}=\Delta\sqrt f
\]
with a nonnegative initial datum $f_0$. The equation preserves the mass $M=\ix f$ and is such that
\[
\frac d{dt}\(\ix{f\,\log f}-\frac{4\,\pi}M\ix{f\((-\Delta)^{-1}f\)}\)=-\,\frac8M\(\ix{\left|\nabla f^{1/4}\right|^2}\,\ix f-\,\pi\ix{f^{3/2}}\)\,.
\]
According to~\cite{MR1940370}, the Gagliardo-Nirenberg inequality 
\be{GN}
\nrm{\nabla g}2^2\,\nrm g4^4\ge\pi\,\nrm g6^6
\ee
applied to $g=f^{1/4}$ guarantees that the right-hand side is nonpositive. By the general theory of fast diffusion equations (we refer for instance to \cite{MR2282669}), we know that the solution behaves for large values of~$t$ like a self-similar solution, the so-called Barenblatt solution, which is given by $B(t,x):=t^{-2}\,f_\star(x/t)$. As a consequence, we find that
\begin{multline*}
\ix{f_0\,\log f_0}-\frac{4\,\pi}M\ix{f_0\((-\Delta)^{-1}f_0\)}\\
\ge\lim_{t\to+\infty}\ix{B\,\log B}-\frac{4\,\pi}M\ix{B\((-\Delta)^{-1}B\)}=\ix{f_\star\,\log f_\star}-\frac{4\,\pi}M\ix{f_\star\((-\Delta)^{-1}f_\star\)}
\end{multline*}
After an elementary computation, we observe that the above inequality is exactly~\eqref{logHLS0} written for $f=f_0$.

The point is now to adapt this strategy to the case with an external potential. This justifies why we have to introduce a nonlinear diffusion equation with a drift. As we shall see below, the method is insensitive to $\alpha$ and applies when $\alpha>1$ exactly as in the case $\alpha\in(0,1)$. A natural question is whether solutions are regular enough to perform the computations below and in particular if they have a sufficient decay at infinity to allow all kinds of integrations by parts needed by the method. The answer is twofold. First, we can take an initial datum $f_0$ that is as smooth and decaying as $|x|\to+\infty$ as needed, prove the inequality and argue by density. Second, integrations by parts can be justified by an approximation scheme consisting in a truncation of the problem in larger and larger balls. We refer to~\cite{MR2282669} for regularity issues and to~\cite{MR3497125} for the truncation method. In the proof, we will therefore leave these issues aside, as they are purely technical.

\begin{proof}[\emph{\bf Proof of Theorem~\ref{Thm:Main}}] By homogeneity, we can assume that $M=1$ without loss of generality and consider the evolution equation
\[
\frac{\partial f}{\partial t}=\Delta\sqrt f+2\,\sqrt\pi\,\nabla\cdot(x\,f)\,.
\]
1) Using simple integrations by parts, we compute
\[
\ix{\big(1+\log f\big)\,\Delta\sqrt f}=-\,8\ix{\left|\nabla f^{1/4}\right|^2}
\]
and
\[
\ix{\big(1+\log f\big)\,\nabla\cdot(x\,f)}=-\ix{\frac{\nabla f}f\cdot(x\,f)}=-\ix{x\cdot\nabla f}=2\ix f=2\,.
\]
As a consequence, we obtain that
\be{Id1}
\frac d{dt}\ix{f\,\log f}=-\,8\ix{\left|\nabla f^{1/4}\right|^2}+\,8\,\pi\ix{\mu^{3/2}}
\ee
using
\[
\ix{\mu^{3/2}}=\frac1{2\,\sqrt\pi}\,.
\]
2) By elementary considerations again, we find that
\[
4\,\pi\ix{f\,(-\Delta)^{-1}\(\Delta\sqrt f\,\)}=-\,4\,\pi\ix{f^{3/2}}
\]
and
\begin{multline*}
4\,\pi\ix{\nabla\cdot(x\,f)\,(-\Delta)^{-1}f}=-\,4\,\pi\ix{x\,f\cdot\nabla(-\Delta)^{-1}f}\\
\hspace*{3cm}=2\iint_{\R^2\times\R^2}f(x)\,f(y)\,x\cdot\frac{x-y}{|x-y|^2}\,dx\,dy\\
=\iint_{\R^2\times\R^2}f(x)\,f(y)\,(x-y)\cdot\frac{x-y}{|x-y|^2}\,dx\,dy=1
\end{multline*}
where, in the last line, we exchanged the variables $x$ and $y$ and took the half sum of the two expressions. This proves that
\be{Id2}
\frac d{dt}\(4\,\pi\ix{f\((-\Delta)^{-1}f\)}\)=-\,8\,\pi\ix{\(f^{3/2}-\,\mu^{3/2}\)}\,.
\ee
3) We observe that
\[
\mu(x)=\frac 1{\pi\,\big(1+|x|^2\big)^2}=e^{-V(x)}
\]
solves
\be{Vmu}
\Delta V=-\Delta\log \mu=8\,\pi\,\mu
\ee
and, as a consequence,
\[
\ix{V\,\Delta\sqrt f}=\ix{\Delta V\,\sqrt f}=8\,\pi\ix{\mu\,\sqrt f}\,.
\]
Since
\begin{multline*}
2\,\sqrt\pi\ix{V\,\nabla\cdot(x\,f)}=-\,2\,\sqrt\pi\ix{f\,x\cdot\nabla V}=-\,8\,\sqrt\pi\ix{\frac{|x|^2}{1+|x|^2}\,f}\\
=-\,8\,\sqrt\pi+8\,\sqrt\pi\ix{\frac f{1+|x|^2}}=-\,8\,\sqrt\pi+\,8\,\pi\ix{\sqrt\mu\,f}\,,
\end{multline*}
we conclude that
\be{Id3}
\frac d{dt}\ix{f\,V}=8\,\pi\ix{\(\mu\,\sqrt f+\sqrt\mu\,f-\,2\,\mu^{3/2}\)}\,.
\ee

Let us define
\[
\mathcal F[f]:=\ix{f\,\log f}+\alpha\ix{V\,f}+(1-\alpha)\,\(1+\log\pi\)+2\,(1-\alpha)\iint_{\R^2\times\R^2}f(x)\,f(y)\,\log|x-y|\,dx\,dy\,.
\]
Collecting~\eqref{Id1},~\eqref{Id2} and~\eqref{Id3}, we find that
\[
\frac d{dt}\mathcal F[f(t,\cdot)]=-\,8\(\ix{\left|\nabla f^{1/4}\right|^2}-\,\pi\ix{f^{3/2}}\)-\,8\,\pi\,\alpha\ix{\(f^{3/2}-\mu\,\sqrt f-\sqrt\mu\,f+\mu^{3/2}\)}\,.
\]
Notice that 
\[
\ix{\(f^{3/2}-\mu\,\sqrt f-\sqrt\mu\,f+\mu^{3/2}\)}=\ix{\varphi\(\frac f\mu\)\,\mu^{3/2}}\quad\mbox{with}\quad\varphi(t):=t^{3/2}-t-\sqrt t+1
\]
and that $\varphi$ is a strictly convex function on $\R^+$ such that $\varphi(1)=\varphi'(1)=0$, so that $\varphi$ is nonnegative. On the other hand, by~\eqref{GN}, we know that 
\[
\ix{\left|\nabla f^{1/4}\right|^2}-\,\pi\ix{f^{3/2}}\ge0
\]
as in the proof of~\cite{MR2745814}. Altogether, this proves that $t\mapsto\mathcal F[f(t,\cdot)]$ is monotone nonincreasing. Hence
\[
\mathcal F[f_0]\ge\mathcal F[f(t,\cdot)]\ge\lim_{t\to+\infty}\mathcal F[f(t,\cdot)]=\mathcal F[f_\star]=0\,.
\]
This completes the proof of~\eqref{Ineq:LogHLS}.\end{proof}

\section{Consequences}\label{Sec:Csq}

\subsection{Proof of Corollary~\texorpdfstring{\ref{Cor:Main}}{Corollary}}\label{Sec:ProofCorollary}

To prove the result of Corollary~\ref{Cor:Main}, we have to establish first that the \emph{free energy} functional $\mathcal F_\beta$ is bounded from below. Instead of using standard variational methods to prove that a minimizer is achieved, we can rely on the flow associated with~\eqref{DD}-\eqref{P}.

\medskip\noindent$\bullet$ \textit{\textbf{Repulsive case.}}
Let us consider the \emph{free energy} functional defined in~\eqref{FreeEnergy} where $\phi$ is given by~\eqref{P} with $\varepsilon=+1$, \emph{i.e.}, $\phi=-\,\frac1{2\,\pi}\,\log|\cdot|*f$.
\begin{lemma}\label{Lem:BdedRep} Let $M>0$ and $\varepsilon=+1$. Then $\mathcal F_\beta$ is bounded from below on the set of the functions $f\in\mathrm L^1_+(\R^2)$ such that $\ix f=M$ if $\beta\ge1+\frac M{8\,\pi}$.\end{lemma}
\begin{proof} With $g=\frac fM$ and $\alpha=1+\frac M{8\,\pi}$, this means that
\begin{multline*}
\frac1M\,\mathcal F_\beta[f]-\log M=\ix{g\,\log g}+\beta\ix{V\,g}-\,\frac M{4\,\pi}\iint_{\R^2\times\R^2}g(x)\,g(y)\,\log|x-y|\,dx\,dy\\
=(\beta-\alpha)\ix{V\,g}+\ix{g\,\log g}+\alpha\ix{V\,g}-\,2\,(\alpha-1)\iint_{\R^2\times\R^2}g(x)\,g(y)\,\log|x-y|\,dx\,dy\\
\ge(\beta-\alpha)\ix{V\,g}-\,(1-\alpha)\,\(1+\log\pi\)
\end{multline*}
according to Theorem~\ref{Thm:Main}; the condition $\beta\ge\alpha$ is enough to prove that $\mathcal F_\beta[f]$ is bounded from below.
\end{proof}

\begin{proof}[\emph{\bf Proof of Corollary~\ref{Cor:Main} with $\varepsilon=+1$}]
Let us consider a smooth solution of~\eqref{DD}-\eqref{P}. We refer to~\cite{Nernst-Planck} for details and to~\cite{MR1842428} for similar arguments in dimension $d\ge3$. According to~\eqref{EntropyProduction}, $f$ converges as $t\to+\infty$ to a solution of
\[
\nabla\log f+\beta\,\nabla V+\nabla\phi=0\,.
\]
Notice that this already proves the existence of a stationary solution. The equation can be solved as
\[
f=M\,\frac{e^{-\beta\,V-\phi}}{\ix{e^{-\beta\,V-\phi}}}
\]
after taking into account the conservation of the mass. With~\eqref{P}, the problem is reduced to solving
\[
-\,\Delta\psi=M\(\frac{e^{-\gamma\,V-\psi}}{\ix{e^{-\gamma\,V-\psi}}}-\mu\)\,,\quad\psi=\(\beta-\gamma\)V+\phi\,,\quad\gamma=\beta-\frac M{8\,\pi}
\]
using~\eqref{Vmu}. It is a critical point of the functional $\psi\mapsto\mathcal J_{M,\gamma}[\psi]:=\frac12\ix{|\nabla\psi|^2}+M\ix{\psi\,\mu}+M\,\log\(\ix{e^{-\gamma\,V-\psi}}\)$. Such a functional is strictly convex as, for instance, in~\cite{MR1115290,MR1677677}. We conclude that $\psi$ is unique up to an additional constant.\end{proof}

\noindent$\bullet$ {\textit{\textbf{Attractive case.}}}
Let us consider the \emph{free energy} functional~\eqref{FreeEnergy} $\mathcal F_\beta$ where $\phi$ is given by~\eqref{P} with $\varepsilon=-1$, \emph{i.e.}, $\phi=\frac1{2\,\pi}\,\log|\cdot|*f$. Inspired by~\cite{MR2103197}, we have the following estimate.
\begin{lemma}\label{Lem:BdedAtt} Let $\varepsilon=-1$. Then $\mathcal F_\beta$ is bounded from below on the set of the functions $f\in\mathrm L^1_+(\R^2)$ such that $\ix f=M$ if $M\le8\,\pi$ and $\beta\ge1-\frac M{8\,\pi}$. It is not bounded from below if $M>8\,\pi$.\end{lemma}
\begin{proof} With $g=\frac fM$ and $\alpha=1-\frac M{8\,\pi}$, Theorem~\ref{Thm:Main} applied to
\begin{multline*}
\frac1M\,\mathcal F_\beta[f]-\log M\\
=(\beta-\alpha)\ix{V\,g}+\ix{g\,\log g}+\alpha\ix{V\,g}+\,2\,(1-\alpha)\iint_{\R^2\times\R^2}g(x)\,g(y)\,\log|x-y|\,dx\,dy\\
\ge(\beta-\alpha)\ix{V\,g}-\,(1-\alpha)\,\(1+\log\pi\)
\end{multline*}
proves that the free energy is bounded from below if $M\le8\,\pi$ and $\beta\ge\alpha$. On the other hand, if $f_\lambda(x):=\lambda^{-2}\,f(\lambda^{-1}\,x)$ and $M>8\,\pi$, then
\[
\mathcal F_\beta[f_\lambda]\sim\,2\,M\(\frac M{8\,\pi}-1\)\log\lambda\to-\infty\quad\mbox{as}\quad\lambda\to0_+\,,
\]
which proves that $\mathcal F_\beta$ is not bounded from below.\end{proof}

\begin{proof}[\emph{\bf Proof of Corollary~\ref{Cor:Main} with $\varepsilon=-1$}] The proof goes as in the case $\beta=0$. We refer to~\cite{MR2226917} and leave details to the reader.\end{proof}

\begin{remark} Let us notice that $\mathcal F_\beta$ is unbounded from below if $\beta<0$. This follows from the observation that $\lim_{|y|\to\infty}\mathcal F_\beta[f_y]=-\,\infty$ where $f_y(x)=f(x+y)$ for any admissible $f$.\end{remark}

\subsection{Duality}\label{Sec:Onofri}

When $\alpha>1$, we can write a first inequality by considering the \emph{repulsive case} in the proof of Corollary~\ref{Cor:Main} and observing that
\[
\mathcal J_{M,\gamma}[\psi]\ge\min\mathcal J_{M,\gamma}
\]
where $\psi\in\mathrm W^{2,1}_{\rm{loc}}(\R^2)$ is such that $\ix{(\Delta\psi)}=0$ and the minimum is taken on the same set of functions.

When $\alpha\in[0,1)$, it is possible to argue by duality as in~\cite[Section~2]{MR2996772}. Since $f_\star$ realizes the equality case in~\eqref{Ineq:LogHLS}, we know that
\[
\ix{f_\star\,\log\(\frac{f_\star}M\)}+\alpha\ix{V\,f_\star}+M\,(1-\alpha)\,\(1+\log\pi\)=\frac 2M\,(\alpha-1)\iint_{\R^2\times\R^2}f_\star(x)\,f_\star(y)\,\log|x-y|\,dx\,dy
\]
and, using the fact that $f_\star$ is a critical point of the difference of the two sides of~\eqref{Ineq:LogHLS}, we also have that
\[
\ix{\log\(\frac f{f_\star}\)(f-f_\star)}+\alpha\ix{V\,(f-f_\star)}=\frac 4M\,(\alpha-1)\iint_{\R^2\times\R^2}\big(f(x)-f_\star(x)\big)\,f_\star(y)\,\log|x-y|\,dx\,dy\,.
\]
By subtracting the first identity to~\eqref{Ineq:LogHLS} and adding the second identity, we can rephrase~\eqref{Ineq:LogHLS} as
\[
\mathcal F_{(1)}[f]:=\ix{f\,\log\(\frac f{f_\star}\)}\ge\frac{4\,\pi}M\,(1-\alpha)\int_{\R^2}(f-f_\star)\,(-\Delta)^{-1}(f-f_\star)\,dx:=\mathcal F_{(2)}[f]\,.
\]
Let us consider the Legendre transform
\[
\mathcal F_{(i)}^*[g]:=\sup_f\(\ix{g\,f}-\mathcal F_{(i)}[f]\)
\]
where the supremum is restricted to the set of the functions $f\in\mathrm L^1_+(\R^2)$ such that $M=\ix f$. After taking into account the Lagrange multipliers associated with the mass constraint, we obtain that
\[
M\,\log\(\ix{e^{\,g-\,V}}\)=\mathcal F_{(1)}^*[g]\le\frac M{16\,\pi\,(1-\alpha)}\ix{|\nabla g|^2}+M\ix{g\,e^{-V}}=\mathcal F_{(2)}^*[g]\,.
\]
We can get rid of $M$ by homogeneity and recover the standard Euclidean form of the Onofri inequality in the limit case as $\alpha\to0_+$, which is clearly the sharpest one for all possible $\alpha\in[0,1)$.

\subsection{Extension to general confining potentials with critical asymptotic growth}\label{Sec:confinement}

As a concluding observation, let us consider a general potential $W$ on $\R^2$ such that
\be{W}\tag{$\mathcal H_W$}
W\in C(\R^2)\quad\mbox{and}\quad\lim_{|x|\to+\infty}\frac{W(x)}{V(x)}=\beta
\ee
and the associated \emph{free energy functional}
\[
\mathcal F_{\beta,W}[f]:=\ix{f\,\log f}+\beta\ix{W\,f}+\frac12\ix{\phi\,f}
\]
where $\phi$ is given in terms of $f>0$ by~\eqref{P}. With previous notations, $\mathcal F_\beta=\mathcal F_{\beta,V}$. Our last result is that the asymptotic behaviour obtained from~\eqref{W} is enough to decide whether $\mathcal F_{\beta,W}$ is bounded from below or not. The precise result goes as follows.
\begin{theorem}\label{Thm:confinement} Under Assumption~\eqref{W}, $\mathcal F_{\beta,W}$ defined as above is bounded from below if either $\varepsilon=+1$ and $\beta>1+\frac M{8\,\pi}$, or $\varepsilon=-1$, $\beta>1-\frac M{8\,\pi}$ and $M\le8\,\pi$. The result is also true in the limit case if $(W-\beta\,V)\in\mathrm L^\infty(\R^2)$ and either $\varepsilon=+1$ and $\beta=1+\frac M{8\,\pi}$, or $\varepsilon=-1$, $\beta=1-\frac M{8\,\pi}$ and $M\le8\,\pi$.\end{theorem}
\begin{proof} If $(W-\beta\,V)\in\mathrm L^\infty(\R^2)$, we can write that
\[
\mathcal F_{\beta,W}[f]\ge\mathcal F_\beta[f]-\,M\,\nrm{W-\beta\,V}{\mathrm L^\infty(\R^2)}\,.
\]
This completes the proof in the limit case. Otherwise, we redo the argument using $\tilde\beta\,V-\(\tilde\beta\,V-W\)_+$ for some $\tilde\beta\in(0,\beta)$ if $\varepsilon=-1$, and for some $\tilde\beta\in\(1+\frac M{8\,\pi},\beta\)$ if $\varepsilon=+1$.\end{proof}

\addtolength{\textheight}{1cm}
\noindent{\bf Acknowledgment:} {\small This work has been partially supported by the Project EFI (ANR-17-CE40-0030) of the French National Research Agency (ANR). The authors thank L.~Jeanjean who pointed them some typographical errors. \copyright\,2019 by the authors. This paper may be reproduced, in its entirety, for non-commercial purposes.}

\end{document}